\documentclass[a4paper,11pt]{amsart}
\usepackage{mathrsfs}
\usepackage{bbm}
\usepackage{latexsym, amssymb,amsmath,amsthm}
\usepackage[all, knot]{xy}
\xyoption{arc}

\def \To{\longrightarrow}

\def \gr{\operatorname{gr}}
\def \id{\operatorname{id}}

\def \ord{\operatorname{ord}}

\def \unit{\epsilon}
\def \C{\mathcal{C}}
\def \D{\Delta}

\def \e{\varepsilon}

\def \R{\mathcal{R}}

\def \Z{\mathbb{Z}}

\numberwithin{equation}{section}

\newtheorem{theorem}{Theorem}[section]
\newtheorem{lemma}[theorem]{Lemma}
\newtheorem{proposition}[theorem]{Proposition}
\newtheorem{corollary}[theorem]{Corollary}

\begin{document}
\title[QUIVER-THEORETIC QUASITRIANGULARITY OF HOPF ALGEBRAS]
{ON QUIVER-THEORETIC DESCRIPTION FOR QUASITRIANGULARITY OF HOPF
ALGEBRAS}
\author{Hua-Lin Huang}
\address{School of Mathematics, Shandong University, Jinan 250100, China} \email{hualin@sdu.edu.cn}
\author{Gongxiang Liu}
\address{Department of Mathematics, Nanjing University, Nanjing 210093, China} \email{gxliu@nju.edu.cn}
\date{}
\maketitle

\begin{abstract}
This paper is devoted to the study of the quasitriangularity of Hopf
algebras via Hopf quiver approaches. We give a combinatorial
description of the Hopf quivers whose path coalgebras give rise to
coquasitriangular Hopf algebras. With a help of the quiver setting,
we study general coquasitriangular pointed Hopf algebras and obtain
a complete classification of the finite-dimensional ones over an
algebraically closed field of characteristic 0.

\vskip 10pt

\noindent{\bf 2000 MR Subject Classification} \ \ 16W30, 16W35,
16G20
\end{abstract}

\section{introduction}

Quasitriangular Hopf algebras were introduced by Drinfeld, which
play a crucial role in his theory of quantum groups \cite{d1}. From
the viewpoint of representation theory, quasitriangular Hopf
algebras are close relatives of groups and Lie algebras, as their
representations form braided tensor categories. The construction and
classification of such Hopf algebras have attracted much attention
since they appeared. However, the general classification problem is
still widely open.

As is well-known, quiver methods are very useful in constructing
algebras and studying their representations. Very nice quiver
settings for elementary and pointed Hopf algebras have been built in
various works \cite{cr1,gs,cr2,voz} and shown their advantage in
classifying some interesting classes of Hopf algebras as well as
their representations, see for instance
\cite{c,g,mont2,chyz,ll,l,hl} and other related works. To us, it
seems sensible to ask: is it possible to describe Drinfeld's
quasitriangularity of Hopf algebras via combinatorial properties of
Hopf quivers introduced by Cibils and Rosso \cite{cr2}?

Our main aim is to provide such description and, by making use of
it, to contribute to the classification and representation theory of
quasitriangular Hopf algebras. In this paper we focus on
quasitriangular elementary Hopf algebras, or more precisely the dual
situation coquasitriangular pointed Hopf algebras. Recall that, for
a Hopf algebra, by elementary we mean the underlying algebra is
finite-dimensional and its simple modules are 1-dimensional, while
by pointed we mean the simple comodules of the underlying coalgebra
are 1-dimensional.

In some sense, what we are studying is the simplest class of
quasitriangular Hopf algebras. According to their proximity to
groups, quasitriangular elementary Hopf algebras should be viewed as
the counterpart of finite abelian groups, as they are
finite-dimensional and their simple modules are all 1-dimensional.
In this point of view, the class of quasitriangular pointed Hopf
algebras lies in the other extreme and so more complicated. Their
study is left for future work.

We build a quiver framework for our object and use the combinatorial
properties to study the quasitriangularity. Since the quiver
approach from the coalgebraic side is more convenient and admits a
wider scope (including infinite-dimensional case), and the dual of
quasitriangular elementary Hopf algebras are coquasitriangular
pointed, so we always work mainly on the latter and mention briefly
the dual cases for the former. First we determine the Hopf quivers
whose path coalgebras give rise to coquasitriangular Hopf algebras,
then we give a Gabriel type theorem for general coquasitriangular
pointed Hopf algebras, and finally we provide some examples to
elucidate the quiver setting.

We also apply the constructed quiver setting to classify
finite-dimensional coquasitriangular pointed Hopf algebras over an
algebraically closed field of characteristic 0. First we prove that
such Hopf algebras are generated by group-like and skew-primitive
elements. The proof relies on the result of Andruskiewitsch, Etingof
and Gelaki \cite{aeg} for the cotriangular case as well as an
observation motivated by our quiver setting. This partially confirms
a conjecture of Andruskiewitsch and Schneider \cite{as2} which says
that a finite-dimensional pointed Hopf algebras is generated by its
group-like and skew-primitive elements. Then we give an explicit
description of finite-dimensional coquasitriangular pointed Hopf
algebras via generators and defining relations by making use of the
quiver setting along with the lifting theorem of quantum linear
spaces due to Andruskiewitsch and Schneider \cite{as1}.

Throughout the paper, we work over a field $k.$ Vector spaces,
algebras, coalgebras, linear mappings, and unadorned $\otimes$ are
over $k.$ The readers are referred to \cite{d2,kassel} for general
knowledge of (co)quasitriangular Hopf algebras, and to \cite{ass}
for that of quivers and their applications to algebras and
representation theory.

\section{(co)quasitriangular hopf algebras and hopf quivers}

This section is devoted to some preliminaries.

\subsection{}
A quasitriangular Hopf algebra is a pair $(H,R),$ where $H$ is a
Hopf algebra and $R \in H \otimes H$ is an invertible element
satisfying
\begin{gather}
(\D \otimes \id)(R) = R_{13}R_{23}, \\
(\id \otimes \D)(R) = R_{13}R_{12}, \\
\D^o(h)= R \D(h) R^{-1}, \forall \ h \in H,
\end{gather}
where $\D^o$ is the opposite coproduct of $H.$ Writing $R = \sum
R_{(1)} \otimes R_{(2)},$ the notation used is \[ R_{ij} = \sum 1
\otimes \cdots \otimes R_{(1)} \otimes 1 \otimes \cdots \otimes
R_{(2)} \otimes \cdots \otimes 1, \] the element of $H \otimes H
\otimes \cdots \otimes H$ which is $R$ in the $i$th and $j$th
factors. We call $R$ a universal $R$-matrix, or a quasitriangular
structure, of $H.$ A quasitriangular Hopf algebra $(H,R)$ is called
triangular, if \begin{equation} RR_{21}=1 \otimes 1. \end{equation}

Dually, a coquasitriangular Hopf algebra $(H,\R)$ is a Hopf algebra
$H$ and a convolution-invertible linear function $\R: H \otimes H
\To k$ such that
\begin{gather}
\R(ab , c) = \R(a , c_1) \R(b , c_2), \\
\R(a , bc) = \R(a_1 , c) \R(a_2 , b), \\
b_1a_1\R(a_2 , b_2) = \R(a_1 , b_1)a_2b_2
\end{gather}
for all $a,b,c \in H.$ Here and below we use the Sweedler sigma
notation $\D(a)=a_1 \otimes a_2$ for the coproduct. The function
$\R$ is called a universal $R$-form, or a coquasitriangular
structure, of $H.$ Let $\bar{\R}$ denote the convolution-inverse of
$\R.$ Then the last condition (2.7) is equivalent to
\begin{equation}
ba = \R(a_1 , b_1) a_2b_2 \bar{\R}(a_3 , b_3)
\end{equation}
for all $a,b \in H.$ A coquasitriangular Hopf algebra $(H,\R)$ is
called cotriangular, if \begin{equation} \R(a_1 , b_1)\R(b_2 ,a_2) =
\e(a)\e(b) \end{equation} for all $a,b \in H.$

\subsection{}
A quiver is a quadruple $Q=(Q_0,Q_1,s,t),$ where $Q_0$ is the set of
vertices, $Q_1$ is the set of arrows, and $s,t:\ Q_1 \longrightarrow
Q_0$ are two maps assigning respectively the source and the target
for each arrow. A path of length $l \ge 1$ in the quiver $Q$ is a
finitely ordered sequence of $l$ arrows $a_l \cdots a_1$ such that
$s(a_{i+1})=t(a_i)$ for $1 \le i \le l-1.$ By convention a vertex is
said to be a trivial path of length $0.$

The path coalgebra $kQ$ is the $k$-space spanned by the paths of $Q$
with counit and comultiplication maps defined by $\e(g)=1, \ \D(g)=g
\otimes g$ for each $g \in Q_0,$ and for each nontrivial path $p=a_n
\cdots a_1, \ \e(p)=0,$
\begin{equation}
\D(a_n \cdots a_1)=p \otimes s(a_1) + \sum_{i=1}^{n-1}a_n \cdots
a_{i+1} \otimes a_i \cdots a_1 \nonumber + t(a_n) \otimes p \ .
\end{equation}
The length of paths gives a natural gradation to the path coalgebra.
Let $Q_n$ denote the set of paths of length $n$ in $Q,$ then
$kQ=\oplus_{n \ge 0} kQ_n$ and $\D(kQ_n) \subseteq
\oplus_{n=i+j}kQ_i \otimes kQ_j.$ Clearly $kQ$ is pointed with the
set of group-likes $G(kQ)=Q_0,$ and has the following coradical
filtration $$ kQ_0 \subseteq kQ_0 \oplus kQ_1 \subseteq kQ_0 \oplus
kQ_1 \oplus kQ_2 \subseteq \cdots.$$ Hence $kQ$ is coradically
graded.

\subsection{}
According to cibils and Rosso \cite{cr2}, a quiver $Q$ is said to be
a Hopf quiver if the corresponding path coalgebra $kQ$ admits a
graded Hopf algebra structure. Hopf quivers can be determined by
ramification data of groups. Let $G$ be a group, $\C$ the set of
conjugacy classes. A ramification datum $R$ of the group $G$ is a
formal sum $\sum_{C \in \C}R_CC$ of conjugacy classes with
coefficients in $\mathbb{N}=\{0,1,2,\cdots\}.$ The corresponding
Hopf quiver $Q=Q(G,R)$ is defined as follows: the set of vertices
$Q_0$ is $G,$ and for each $x \in G$ and $c \in C,$ there are $R_C$
arrows going from $x$ to $cx.$ For a given Hopf quiver $Q,$ the set
of graded Hopf structures on $kQ$ is in one-to-one correspondence
with the set of $kQ_0$-Hopf bimodule structures on $kQ_1.$

The graded Hopf structures are obtained from Hopf bimodules via the
quantum shuffle product \cite{rosso}. Suppose that $Q$ is a Hopf
quiver with a necessary $kQ_0$-Hopf bimodule structure on $kQ_1.$
Let $p \in Q_l$ be a path. An $n$-thin split of it is a sequence
$(p_1, \ \cdots, \ p_n)$ of vertices and arrows such that the
concatenation $p_n \cdots p_1$ is exactly $p.$ These $n$-thin splits
are in one-to-one correspondence with the $n$-sequences of $(n-l)$
0's and $l$ 1's. Denote the set of such sequences by $D_l^n.$
Clearly $|D_l^n|={n \choose l}.$ For $d=(d_1, \ \cdots, \ d_n) \in
D_l^n,$ the corresponding $n$-thin split is written as $dp=((dp)_1,
\ \cdots, \ (dp)_n),$ in which $(dp)_i$ is a vertex if $d_i=0$ and
an arrow if $d_i=1.$ Let $\alpha=a_m \cdots a_1$ and $\beta=b_n
\cdots b_1$ be paths of length $m$ and $n$ respectively. Let $d \in
D_m^{m+n}$ and $\bar{d} \in D_n^{m+n}$ the complement sequence which
is obtained from $d$ by replacing each 0 by 1 and each 1 by 0.
Define an element
$$(\alpha \cdot \beta)_d=[(d\alpha)_{m+n}.(\bar{d}\beta)_{m+n}] \cdots
[(d\alpha)_1.(\bar{d}\beta)_1]$$ in $\c Q_{m+n},$ where
$[(d\alpha)_i.(\bar{d}\beta)_i]$ is understood as the action of
$kQ_0$-Hopf bimodule on $kQ_1$ and these terms in different brackets
are put together linearly by concatenation. In terms of these
notations, the formula of the product of $\alpha$ and $\beta$ is
given as follows:
\begin{equation}
\alpha \cdot \beta=\sum_{d \in D_m^{m+n}}(\alpha \cdot \beta)_d \ .
\end{equation}

\subsection{}

Let $H$ be a pointed Hopf algebra. Denote its coradical filtration
by $\{ H_n \}_{n=0}^{\infty}.$ Define \[\operatorname{gr}(H)=H_0
\oplus H_1/H_0 \oplus H_2/H_1 \oplus \cdots \cdots \] as the
corresponding (coradically) graded coalgebra. Then
$\operatorname{gr}(H)$ inherits from $H$ a coradically graded Hopf
algebra structure (see e.g. \cite{mont1}). The procedure from $H$ to
$\gr H$ is generally called degeneration and the converse is called
lifting by Andruskiewitsch and Schneider in \cite{as1}.

Due to Van Oystaeyen and Zhang \cite{voz}, if $H$ is a coradically
graded pointed Hopf algebra, then there exists a unique Hopf quiver
$Q(H)$ such that $H$ can be realized as a large sub Hopf algebra of
a graded Hopf structure on the path coalgebra $kQ(H).$ Here by
``large" we mean $H$ contains the subspace $kQ(H)_0 \oplus kQ(H)_1.$

\section{coquasitriangular hopf quivers}

In this section we determine the Hopf quivers which give rise to
coquasitriangular Hopf algebras. In addition, some classification
results of the coquasitriangular structures on such quiver Hopf
algebras are given.

\subsection{}
Let $Q$ be a quiver. If the path coalgebra $kQ$ admits a graded
coquasitriangular Hopf algebra structure, then we called $Q$ a
coquasitriangular Hopf quiver. First we give a combinatorial
description of such class of Hopf quivers.

In the following, we need to make use of the classification of Hopf
bimodules over a group algebra obtained in \cite{cr1} by Cibils and
Rosso. That is, for a group $G,$ the category of $kG$-Hopf bimodules
is equivalent to the product of the categories of usual module
categories $\prod_{C \in \C} kZ_C\mathrm{-mod},$ where $\C$ is the
set of conjugacy classes and $Z_C$ is the centralizer of one of the
elements in the class $C \in \C.$ In particular, if $G$ is abelian,
then the category of $kG$-Hopf bimodules is equivalent to $\prod_{g
\in G} kG\mathrm{-mod}.$

\begin{proposition}
Let $Q$ be a quiver. Then $Q$ is a coquasitriangular Hopf quiver if
and only if $Q$ is a Hopf quiver of form $Q(G,R)$ where $G$ is an
abelian group and $R$ a ramification datum.
\end{proposition}

\begin{proof}
Assume that $Q$ is a coquasitriangular Hopf quiver and that
$(kQ,\R)$ is a graded coquasitriangular Hopf algebra. Then in the
first place $Q$ must be a Hopf quiver, say $Q(G,R).$ Let $\R_0$ be
the restriction of $\R$ to $kQ_0 \otimes kQ_0.$ Note that $kQ_0=kG$
is a group algebra and that $(kG,\R_0)$ is a coquasitriangular Hopf
algebra. Therefore by (2.7) we have \[ hg = \R_0(g,h)gh
\bar{\R}_0(g,h) = gh \] for all $g,h \in G.$ This implies that $G$
is an abelian group.

Conversely, assume that $Q$ is the Hopf quiver $Q(G,R)$ of some
abelian group $G$ with respect to a ramification datum $R.$ Then we
can take the $kQ_0$-Hopf bimodule structure on $kQ_1$ which
corresponds to the product of a set of trivial $kG$-modules. This
gives rise to a commutative graded Hopf structure on $kQ.$
Apparently $(kQ, \e \otimes \e)$ is a coquasitriangular Hopf
algebra. This proves that $Q$ is a coquasitriangular Hopf quiver.
\end{proof}

\subsection{}
Let $Q$ be a coquasitriangular Hopf quiver. The proof of the
previous proposition provides only a trivial coquasitriangular Hopf
structure on the path coalgebra $kQ.$ Now we study nontrivial
coquasitriangular structures on $kQ.$ Assume that $Q=Q(G,R).$ We
start with the simplest case. Fix a graded Hopf structure on $kQ$
which is determined by a given $kG$-Hopf bimodule structure on
$kQ_1.$ Assume that $\R$ is a coquasitriangular structure which
concentrates at degree 0, namely, $\R(x,y)=0$ for all homogenous
elements $x,y$ unless they lie in $kQ_0$. The following facts are
straightforward.

\begin{lemma}
Keep the previous assumptions. Then $\R$ is a bicharacter of $G,$
that is, for all $f,g,h \in G,$
\begin{gather}
\R(f , gh)=\R(f , g)\R(f , h), \quad \R(fg, h)=\R(f , h)\R(g , h), \\
\R(f,\unit)=1=\R(\unit,f) \notag
\end{gather}
where $\unit$ is the unit of $G.$ Moreover, in $kQ$ the following
equations hold:
\begin{gather}
\alpha g = \frac{\R(g , t(\alpha))}{\R(g , s(\alpha))}
g \alpha \quad \text{for all $g \in G, \ \alpha \in Q_1,$} \\
\beta \alpha = \frac{\R(t(\alpha) , t(\beta))}{\R(s(\alpha) ,
s(\beta))} \alpha \beta \quad \text{for all $\alpha,\beta \in Q_1.$}
\end{gather}
\end{lemma}

Here and below, the product notation ``$\cdot$" is omitted whenever
it is clear according to the context.

The preceding lemma indicates that the coquasitriangular structure
$\R$ and the $kG$-Hopf bimodule structure on $kQ_1$ are mutually
determined.

\subsection{}
Let $Q=Q(G,R)$ be a coquasitriangular Hopf quiver. Note that $Q$ is
connected if and only if the union of the conjugacy classes with
non-zero coefficients in the ramification datum $R$ generates $G.$
In general, let $Q(\unit)$ denote the connected component containing
the unit $\unit \in G,$ then $Q(\unit)_0$ is a normal subgroup of
$G$ and as a quiver $Q$ is equivalent to $\cup_{x \in G/Q(\unit)_0}
Q(x)$ where $Q(x)$ is identical to $Q(\unit)$ for all $x.$
Similarly, coquasitriangular Hopf structures on $kQ$ can also be
decomposed as a crossed product (see e.g. \cite{mont1}) of
$kQ(\unit),$ which inherits a sub coquasitriangular structure of
$kQ,$ with the quotient group algebra $kG/Q(\unit)_0.$ This enables
us to work on only connected Hopf quivers.

Now we are ready to state the main result of this section.

\begin{theorem}
Let $Q=Q(G,R)$ be a connected coquasitriangular Hopf quiver. Then
the complete list of coquasitriangular graded Hopf algebra
structures $(kQ,\R)$ with $Q_0$ equal to $G$ and $\R$ concentrating
at degree 0 is in one-to-one correspondence with the set of
skew-symmetric bicharacters of $G.$ Here by ``skew-symmetric" it is
meant $ \R(g , h)\R(h , g)=1 $ for all $g,h \in G.$
\end{theorem}

\begin{proof}
First assume that $(kQ,\R)$ is a coquasitriangular Hopf algebra with
$Q_0$ equal to $G$ and $\R$ concentrating at degree 0. Then by
restricting to $G,$ we have showed in Lemma 3.2 that $\R$ is a
bicharacter of $G$ satisfying (3.2)-(3.3). Assume that $R=\sum_{g
\in G} R_gg.$ Since $Q(G,R)$ is connected, we have $<g \in G | R_g
\ne 0>=G.$ Suppose that $R_g \ne 0 \ne R_h,$ then in $Q$ there are
arrows starting from the unit $\unit$ of $G,$ say $\alpha: \unit \To
g$ and $\beta: \unit \To h.$ Then by (2.10) and (3.2)-(3.3) we have
\begin{eqnarray*}
\beta \alpha &=& [\beta g][\alpha] + [h \alpha][\beta] = \R(g
, h)[g \beta][\alpha] + [h \alpha][\beta] \\
&=& \R(g , h)\alpha \beta = \R(g , h)[\alpha h][\beta] + \R(g , h)[g \beta][\alpha] \\
&=& \R(g , h)\R(h , g)[h \alpha][\beta] + \R(g , h)[g \beta][\alpha]. \\
\end{eqnarray*}
This implies that $\R(g , h)\R(h , g)=1.$ Since both $g$ and $h$
vary a generating set of $G,$ therefore we have $\R(a , b)\R(b ,
a)=1$ for all $a,b \in G.$

Conversely, assume that $\R$ is a skew-symmetric bicharacter of $G.$
For each arrow in $Q$ with source $\unit,$ say $\alpha: \unit \To
g,$ let $h \triangleright \alpha := \R(g , h)\alpha.$ This defines a
one-dimensional left $G$-module on $k\alpha.$ Denote by $V$ the
vector space spanned by all arrows with source $\epsilon$ and
understand a $G$-module structure on it by defining $G$-action on
each arrow as above. Now by the aforementioned results of Cibils and
Rosso \cite{cr1,cr2}, we can extend the left $G$-module $V$ to a
$kG$-Hopf bimodule on $kQ_1$ and this gives rise to a unique graded
Hopf algebra structure on the path coalgebra $kQ.$ We extend $\R$
trivially to be a function on $kQ \otimes kQ$ such that $\R(x ,
y)=0$ whenever one of the homogeneous elements $x,y$ lies out of
$kQ_0.$ Now we claim that $(kQ,\R)$ is a coquasitriangular Hopf
algebra. The axioms (2.5)-(2.6) are direct from the definition of
$\R.$ It remains to verify (2.8). We need to show that the following
equation holds
\[ \beta \alpha = \R(t(a_m) , t(b_n)) \alpha \beta \bar{\R}(s(a_1) , s(b_1)) \]
for all paths $\alpha=a_m \cdots a_1,\ \beta=b_n \cdots b_1.$ Here
we use the convention: if $m=0,$ then $\alpha \in Q_0$ and
$t(\alpha)=\alpha=s(\alpha).$ For $\alpha, \beta \in Q_0,$ the
equation is obvious. For $\alpha \in Q_1$ an arrow with source
$\epsilon$ and $\beta \in Q_0,$ the equation follows from the
definition of the action $``h \triangleright \alpha"$ which
corresponds to conjugation $``h \alpha h^{-1}"$ in the $kG$-Hopf
bimodule. Now let $\alpha: g \To h$ be an arbitrary arrow and $f \in
G.$ Note that $g^{-1}\alpha$ is an arrow from $\unit$ to $g^{-1}h.$
Denote it by $\tilde{\alpha}.$ Then we have by (3.1) and the
previous case that
\[ f \alpha = gf \tilde{\alpha} = g \R(g^{-1}h , f)
\tilde{\alpha} f = \R(h , f) \alpha f \bar{\R}(g , f).\] Finally let
$\alpha=a_m \cdots a_1,\ \beta=b_n \cdots b_1$ be arbitrary paths.
Then we have by the preceding cases and the product formula (2.10)
that
\begin{eqnarray*}
 \beta \alpha &=& \sum_{d \in D_n^{m+n}} [(d\beta)_{m+n} (\bar{d}\alpha)_{m+n}]
 \cdots [(d\beta)_1 (\bar{d}\alpha)_1] \\
 &=& \sum_{d \in D_n^{m+n}} [\frac{\R(t((\bar{d}\alpha)_{m+n}) ,
 t((d\beta)_{m+n}))}{\R(s((\bar{d}\alpha)_{m+n}) ,
 s((d\beta)_{m+n}))} (\bar{d}\alpha)_{m+n} (d\beta)_{m+n}] \cdots \\
 & & [\frac{\R(t((\bar{d}\alpha)_1) , t((d\beta)_1))}{R(s((\bar{d}\alpha)_1)
 , s((d\beta)_1))} (\bar{d}\alpha)_1 (d\beta)_1] \\
 &=& \frac{\R(t(a_m) , t(b_n))}{\R(s(a_1) , s(b_1))} \sum_{d \in D_n^{m+n}}
 [(\bar{d}\alpha)_{m+n} (d\beta)_{m+n}] \cdots [(\bar{d}\alpha)_1 (d\beta)_1] \\
 &=& \frac{\R(t(a_m) , t(b_n))}{\R(s(a_1) , s(b_1))} \sum_{d \in D_m^{m+n}}
 [(d\alpha)_{m+n} (\bar{d}\beta)_{m+n}] \cdots [(d\alpha)_1 (\bar{d}\beta)_1] \\
 &=& \R(t(a_m) , t(b_n)) \alpha \beta \bar{\R}(s(a_1) , s(b_1)).
\end{eqnarray*} In the third equality we have used the fact
$t((d\beta)_i)=s((d\beta)_{i+1})$ for $i=1, \cdots, m+n-1.$ Now we
are done.
\end{proof}

\subsection{}
In the following we show that the situation ``$\R$ concentrates at
degree 0" can always be realized as restricting a general
coquasitriangular structure. Let $(kQ,\R)$ be a coquasitriangular
Hopf algebra, define $\R_{0}$ by $\R_{0}|_{kQ_{0}\otimes
kQ_{0}}:=\R|_{kQ_{0}\otimes kQ_{0}}$ and $\R_{0}|_{kQ_{i}\otimes
kQ_{j}}:=0$ if $i+j> 0$. Clearly, $\R _{0}$ concentrates at degree
0.

\begin{proposition}
Let $(kQ,\R)$ be a coquasitriangular Hopf algebra, then  $(kQ, \R
_{0})$ is coquasitriangular too.
\end{proposition}
\begin{proof}  We only show
equation (2.7) for $\R_{0}$ since (2.5) and (2.6) are obvious. Let
$a, b$ be two paths in $kQ$ of lengths $m,n$ respectively, then by
$(kQ,\R)$ is coquasitriangular we have
$$b_1a_1\R(a_2 , b_2) = \R(a_1 , b_1)a_2b_2.$$
Denote the linear combination of all paths of length $m+n$ appearing
in the left (resp. right) side of the above equality by
$L_{m+n}(a,b)$ (resp. $R_{m+n}(a,b)$). Therefore,
$$L_{m+n}(a,b)=R_{m+n}(a,b).$$
Note that  $L_{m+n}(a,b)=b_1 a_1 \R _{0}(a_2 , b_2)$ and
$R_{m+n}(a,b)=\R_{0}(a_1 , b_1)a_2 b_2$, we get what we want.
\end{proof}

\subsection{}
Now the following consequences are immediate.

\begin{corollary}
Let $Q=Q(G,R)$ be a coquasitriangular Hopf quiver and $(kQ,\R)$ a
graded coquasitriangular Hopf algebra.
\begin{enumerate}
  \item If $Q$ is connected, then $(kQ, \R _{0})$ is a cotriangular Hopf algebra.
  \item Denote by $\R(\unit)$ the restriction of $\R$ on the connected component $Q(\unit)$ of
$Q$ containing the unit element, then the sub coquasitriangular Hopf
algebra $(kQ(\unit),\R _{0}(\unit))$ is cotriangular.
\end{enumerate}
\end{corollary}

Here, it is worthy to note that for graded coquasitriangular Hopf
algebras $(kQ,\R)$ as above the ``connectedness" of $Q$ leads to the
``cotriangularity" of $kQ.$ Accordingly, a general graded
coquasitriangular Hopf algebra $(kQ,\R)$ is a crossed product of a
cotriangular one with a group algebra. Finally, we remark that there
may be many coquasitriangular structures of $kQ$ not concentrating
at degree 0. A complete classification is not known yet.

\section{quiver setting for coquasitriangular pointed hopf algebras}
The aim of this section is to provide a quiver setting for general
coquasitriangular pointed Hopf algebras. In particular, a Gabriel
type theorem is obtained.

\subsection{}
Assume that $(H,\R)$ is a coquasitriangular pointed Hopf algebra.
Let $\gr(H)$ be its coradically graded version as mentioned in
Subsection 2.4 and let $\gr(\R): \gr(H) \otimes \gr(H) \To k$ be the
function defined by (for homogeneous elements $g,h\in \gr(H)$) \[
\gr(\R)(g,h) = \left\{
                                        \begin{array}{ll}
                                          \R(g,h), & \hbox{if $g,h \in H_0$;} \\
                                          0, & \hbox{otherwise.}
                                        \end{array}
                                      \right. \]
Our first observation is that the degeneration $(\gr(H), \gr(\R))$
of $(H, \R)$ does not lose the coquasitriangularity.

\begin{lemma}
Assume that $(H,\R)$ is a coquasitriangular pointed Hopf algebra,
then $(\gr(H),\gr(\R))$ is still a coquasitriangular pointed Hopf
algebra.
\end{lemma}

\begin{proof} Denote the multiplication notation in $\gr(H)$ by $\circ$.
At first, we check the equation (2.5).
Take any three homogenous elements $a,b,c\in \gr(H)$. If $a\not\in
H_{0}$ or $b\not\in H_{0}$, then
$$\gr(\R)(a\circ b,c)=0=\gr(\R)(a,c_{1})\gr(\R)(a,c_{2})$$
by the definition of $\gr(\R)$. If $a,b\in H_{0}$ and $c\not\in
H_{0}$, then clearly either $c_{1}\not\in H_{0}$ or $c_{2}\not\in
H_{0}$. Also, by the definition of $\gr(\R)$, equation (2.5) is
right since both sides are zero. We only need to prove the case
$a,b,c\in H_{0}$, but this is clear since $H_{0}$ is a sub Hopf
algebra of $H$ and hence coquasitriangular. The equation (2.6) can
be proved similarly.

Next let us show the equation (2.7). By the Gabriel type theorem for
graded pointed Hopf algebras \cite{voz} (Theorem 4.5) and pointed
coalgebras \cite{chz} (Theorem 2.1), there exists a unique Hopf
quiver $Q$ such that $\gr(H)$ is a sub Hopf algebra of $kQ$ and $H$
is sub coalgebra of $kQ$. By this, we can consider the elements of
$\gr(H)$ and $H$ as linear combinations of paths. Take $a,b\in H$
and for simplicity we may assume that $a,b$ are paths of length
$m,n$ respectively. Since $H$ is coquasitriangular, one has
$$b_1a_1\R(a_2 , b_2) = \R(a_1 , b_1)a_2b_2.$$
Similar to the proof of Proposition 3.4, we denote the linear
combination of all paths of length $m+n$ appearing in the left
(resp. right) side of above equality by $L_{m+n}(a,b)$ (resp.
$R_{m+n}(a,b)$). So,
$$L_{m+n}(a,b)=R_{m+n}(a,b).$$
Meanwhile, it is not hard to see that we always have
\[ L_{m+n}(a,b)=b_1\circ a_1\gr(\R)(a_2 , b_2), \quad R_{m+n}(a,b)=\gr(\R)(a_1 , b_1)a_2\circ
b_2. \] Now the proof is finished.
\end{proof}

\subsection{}
Our next observation is that, the Gabriel type theorem for pointed
Hopf algebras \cite{voz} can be restricted to the coquasitriangular
situation.

\begin{theorem}
Suppose that $(H,\R)$ is a coquasitriangular pointed Hopf algebra.
Then there exists a unique coquasitriangular Hopf quiver $Q$ such
that $(\gr(H),\gr(\R))$ is isomorphic to a large sub
coquasitriangular Hopf structure of $(kQ,\mathfrak{R}).$
\end{theorem}

\begin{proof}
Let $G$ denote the set of group-like elements of $H.$ Then the
coradical $H_0$ of $H$ is the group algebra $kG.$ By restricting the
function $\R,$ we have a sub coquasitriangular Hopf algebra
$(kG,\R).$ This implies that $G$ is an abelian group and $\R$ is a
bicharacter of it. By the Gabriel type theorem for general pointed
Hopf algebras, there exists a unique Hopf quiver $Q=Q(G,R)$ such
that $\gr(H)$ is isomorphic to a large sub Hopf algebra of some
graded Hopf structure on $kQ$ determined by the $kG$-Hopf bimodule
$H_1/H_0.$ Note that $(\gr(H),\gr(\R))$ is a graded
coquasitriangular Hopf algebra with $\gr(\R)$ concentrating at
degree zero. By the same argument as in the proof of Theorem 3.3, we
can show that the $kG$-Hopf bimodule $H_1/H_0$ is completely
determined by the bicharacter $\gr(\R)$ and so the associated graded
Hopf structure on $kQ$ is coquasitriangular. Let $\mathfrak{R}$ be
the trivial extension of $\gr(\R)$ to the whole $kQ \otimes kQ,$
then $(kQ,\mathfrak{R})$ is coquasitriangular and apparently
$(\gr(H),\gr(\R))$ is isomorphic to a large sub structure of
$(kQ,\mathfrak{R}).$ This completes the proof.
\end{proof}

\subsection{}
According to a generalization of the Cartier-Gabriel decomposition
theorem due to Montgomery \cite{mont2} (Theorem 3.2), the study of
general coquasitriangular pointed Hopf algebras can be reduced to
the connected case, namely their quivers (in the sense of Theorem
4.2) are connected. This can also be seen intuitively by the fact
that general Hopf quivers consisting of copies of identical
connected components.

For a coquasitriangular pointed Hopf algebra $(H,\R),$ let
$(\gr(H),\gr(\R)) \hookrightarrow (kQ,\mathfrak{R})$ be an embedding
as in Theorem 4.2. By $\gr(H)(\epsilon)$ we denote the image in the
connected component $kQ(\unit).$ This is the principal block of
$\gr(H)$ which is a sub coquasitriangular Hopf algebra with
structure $\gr(\R)(\unit)$ obtained by the obvious restriction. The
other blocks can be obtained by multiplying group-like elements.
Then $\gr(H)$ can be recovered as a crossed product of
$\gr(H)(\epsilon)$ by a group algebra. The coquasitriangular
structure can also be recovered by extending that of
$\gr(H)(\epsilon)$ via the formulae (2.5)-(2.6).

Now we have the following directly from Theorem 4.2 and Corollary
3.5.

\begin{corollary}
Let $(H,\R)$ be a coquasitriangular pointed Hopf algebra and
$(\gr(H),\gr(\R))$ its degeneration as before. Then
$(\gr(H)(\epsilon),\gr(\R)(\unit))$ is cotriangular and
$(\gr(H),\gr(\R))$ is a crossed product of
$(\gr(H)(\epsilon),\gr(\R)(\unit))$ by a group algebra. In
particular, if the quiver of $H$ is connected, then
$(\gr(H),\gr(\R))$ is cotriangular.
\end{corollary}

It is interesting to note that, coquasitriangular pointed Hopf
algebras can always be obtained as lifting of cotriangular ones and
crossed product by group algebras if disconnected.

\subsection{}
We include briefly a dual quiver setting for quasitriangular
elementary Hopf algebras. In this case, we should use the path
algebras of finite quivers (i.e., the set consisting of vertices and
arrows is finite) and their admissible quotients to construct
algebras.

Let $Q$ be a finite quiver. The associated path algebra, denoted by
$kQ^a,$ has the same underlying vector space as the path coalgebra
and the multiplication is defined by concatenation of paths. In
fact, the path algebra $kQ^a$ is the graded dual of the path
coalgebra $kQ.$ The path algebra $kQ^a$ admits a graded Hopf
structure if and only if $Q$ is a Hopf quiver, see \cite{cr1,gs}.

For a quasitriangular elementary Hopf algebra $(H,R),$ we can
consider the graded version $(\gr(H),\gr(R))$ induced by its chain
of Jacobson radical. With this, a quiver setting for such class of
Hopf algebras is given in the following. We state the dual of
Proposition 3.1, Theorems 3.3 and 4.2, Corollary 4.3 without further
explanation.

\begin{theorem}
Let $Q$ be a finite quiver.
\begin{enumerate}
  \item Then the path algebra $kQ^a$ admits a graded quasitriangular Hopf
structure if and only if $Q$ is the Hopf quiver $Q(G,S)$ of some
abelian group $G$ with respect to a ramification datum $S.$
  \item Assume that $Q=Q(G,S)$ is a connected Hopf quiver with $G$
abelian. Then the set of graded quasitriangular Hopf algebras
$(kQ^a,R)$ with $kQ^a_0=(kG)^*$ and $R \in kQ^a_0 \otimes kQ^a_0$ is
in one-to-one correspondence with the set of skew-symmetric
bicharacters of $G.$
\end{enumerate}
\end{theorem}

\begin{theorem}
Let $(H,R)$ be a quasitriangular elementary Hopf algebra and
$(\gr(H),\gr(R))$ its radically graded version. Then there exists a
unique Hopf quiver $Q(H)=Q(G,S)$ with $G$ abelian such that $\gr(H)
\cong kQ(H)^a/I$ as graded quasitriangular Hopf algebras, where $I$
is an admissible Hopf ideal.
\end{theorem}

\begin{corollary}
Let $(H,R)$ be a quasitriangular elementary Hopf algebra and
$(\gr(H),\gr(R))$ its radically graded version. Then the principal
block of $\gr(H)$ is a quotient Hopf algebra and is triangular. The
quasitriangular Hopf algebra $(\gr(H),\gr(R))$ can be presented as a
crossed product of its principal block by the dual of a group
algebra in the sense of Schneider \cite{sch}.
\end{corollary}

\section{examples}

In this section, we construct some examples via the quiver setting.
For the convenience of the exposition, we assume in this section
that the ground field $k$ is the field of complex numbers.

\subsection{}
Let $\Z_n=\langle g \ | \ g^n=1\rangle$ denote the finite cyclic
group of order $n.$ Consider the (coquasitriangular) Hopf quiver
$Q(\Z_n,g).$ It is a basic cycle of length $n.$ For each integer $i$
modulo $n,$ let $a_i$ denote the arrow $g^i \longrightarrow
g^{i+1}.$ Let $\R$ be a skew-symmetric bicharacter of $\Z_n.$ Then
it is completely determined by the value $\lambda=\R(g,g)$ which
should satisfy the equation \[ \lambda^2=1 \] by skew-symmetry. In
general, we have $\R(g^i,g^j)=\lambda^{ij}.$

If $n$ is odd, then $\lambda$ must be $1.$ Therefore, there is only
one skew-symmetric bicharacter, and hence only one graded
coquasitriangular Hopf structure concentrating at degree 0 which is
isomorphic to $k[x] \otimes k\Z_n$ as algebra.

If $n$ is even, then $\lambda=\pm 1.$ The case when $\lambda=1$ is
known. When $\lambda=-1,$ then we can associate to the bicharacter a
$k\Z_n$-Hopf bimodule on $kQ(\Z_n,g)_1$ as follows: \[ ga_i=a_{i+1},
\quad a_ig=-a_{i+1} \] for all integer $i$ modulo $n.$ This gives
rise to a cotriangular Hopf algebra on $kQ(\Z_n,g)$ with path
multiplication given by
\begin{equation*}
p_i^l p_j^m= \left\{
                     \begin{array}{ll}
                       0, & \hbox{if $l,m$ are odd;} \\
                       (-1)^{im}p_{i+j}^{l+m}, & \hbox{otherwise.}
                     \end{array}
                   \right.
\end{equation*}
Here the notation $p_i^l$ means the path of length $l$ with source
$g^i.$ In particular, if we consider the sub Hopf algebras generated
by vertices and arrows, we get the generalized Taft algebra
$C_2(n,-1)$ as denoted in \cite{chyz}. It can be presented by
generators $g$ and $x$ with relations \[ g^n=1, \quad x^2=0, \quad
gx=-xg. \] Note that if $n=2$ then it is exactly the well-known
Sweedler's 4-dimensional Hopf algebra. Such Hopf algebras appeared
in the works of Radford \cite{rad1,rad2} in which their
quasitriangular structures were classified.

Now we determine all the coquasitriangular structures on
$C_2(n,-1).$ Clearly it has $\{g^ix^l | i=0,1,\cdots n-1, \ l=0,1\}$
as a basis. Let $\R$ be a coquasitriangular structure. Then we
already know $\R(g^i , g^j)=(-1)^{ij}.$ By \[ 0=\R(g , x^2) =\R(g ,
x) \R(g , x) \] we have $\R(g , x)=0.$ Then by (2.5)-(2.6) it
follows that $\R(g^i , g^jx)=0.$ Similarly, we have $\R(g^ix ,
g^j)=0.$ By using (2.8) we have
\[ 0=x^2=\R(x , x)1-\overline{\R}(x,x)g^2.\] On the other hand, by applying
$\R\bar{\R}=\e \otimes \e$ to $(x,x),$ we can show that
$\R(x,x)=\bar{\R}(x,x).$ That means, \[ \R(x,x)(1-g^2)=0. \] This
forces $\R(x , x)=0$ if $n>2.$ Similarly by (2.5)-(2.6) we have in
this case $\R(g^ix , g^jx)=0.$ That is, the coquasitriangular
structure on $C_2(n,-1)$ ($n > 2$) is unique which concentrates at
degree 0. If $n=2,$ then we have $\R(x,x)=\overline{\R}(x,x)$ and
$\R(x , x)$ can take any value. Assume $\R(x , x)=\nu.$ Then by
(2.5)-(2.6) it is direct to deduce that \[ \R(gx , x)=-\nu, \quad
\R(x , gx)=\nu, \quad \R(gx , gx)=\nu. \] Of course this is known
from Radford's calculation \cite{rad1} by the fact that the
Sweedler's 4-dimensional Hopf algebra is self-dual.

We remark that the dual of $C_2(n,-1)$ can be presented by
$kQ(\Z_n,g)^a/J^2$ where $J$ is the ideal generated by the set of
arrows. Therefore the elementary Hopf algebra $kQ(\Z_n,g)^a/J^2$ ($n
>2$) has a unique quasitriangular structure, while $kQ(\Z_2,g)^a/J^2$ has a 1-parameter family of
quasitriangular structures. This approach via path algebras was used
previously by Cibils in \cite{c} for the quiver $Q(\Z_n,g).$

\subsection{}
Now we consider the infinite cyclic group $Z=\langle g\rangle$ and
the Hopf quiver $Q(Z,g)$ which is a linear chain. The group $Z$ has
two skew-symmetric bicharacters, namely, $\R(g^i , g^j)=(\pm
1)^{ij}.$ For the trivial bicharacter, the corresponding
coquasitriangular Hopf algebra on $kQ(Z,g)$ is isomorphic to $k[x]
\otimes kZ$ as algebra. For the nontrivial bicharacter, the
coquasitriangular Hopf algebra on $kQ(Z,g)$ has multiplication
formula similar to (5.1). The sub Hopf algebra generated by vertices
and arrows can be presented by generators $g, \ g^{-1}$ and $x$ with
relations \[ gg^{-1}=1=g^{-1}g, \quad x^2=0, \quad gx=-xg. \] By
arguments in the same manner as Subsection 5.1, one can show that it
has a unique coquasitriangular structure concentrating at degree 0.

As a direct consequence of the previous two examples, we can see
that quantized enveloping algebras and small quantum groups are not
coquasitriangular since their subalgebras $U_q(sl_2)^{\ge 0}$ and
$u_q(sl_2)^{\ge 0}$ have $Q(Z,g)$ and $Q(\Z_n,g)$ respectively as
their quivers (see for instance \cite{lus,hsaq1}) but with different
algebra structures from the coquasitriangular ones we have just
classified.

\subsection{}
Note that basic cycles and the linear chain are minimal connected
Hopf quivers and are basic ingredients of the general. The previous
examples classify all the the possible coquasitriangular structures
on these building blocks of general coquasitriangular Hopf quivers.
Next we will consider those examples which are compatible gluing of
them.

Firstly we consider the Hopf quiver $Q=Q(\Z_n,mg).$ It is a
multi-cycle. To avoid the trivial case, we assume $n=2l, \ m>1$ and
$\R$ is the nontrivial skew-symmetric bicharacter of $\Z_n.$ Let
$x_i \ (i=1, \cdots, m)$ denote the arrows with source $1.$ The we
have the following multiplication formulae in $(kQ,\R):$
\[ x_i^2=0, \quad x_ix_j=-x_jx_i, \quad gx_i=-x_ig. \] The sub Hopf
algebra, denote by $E(n,m),$ generated by vertices and arrows can be
presented by generators $g$ and $x_i \ (i=1, \cdots, m)$ with
additional relations $g^n=1.$ Therefore the algebra has a PBW type
basis
\[ \{g^ix_1^{\sigma_1} \cdots x_m^{\sigma_m} | i=0,1,\cdots,n-1; \
\sigma_j=0,1 \}. \] It is not difficult to determine all the
coquasitriangular structures. Let $\R$ be a coquasitriangular
structure. As in Subsection 5.1, we have \[ \R(g^i , g^j)=(-1)^{ij},
\quad \R(g^i , g^jx_k)=0=\R(g^ix_k , g^j)\] and \[ \quad \R(x_i ,
x_j)(1-g^2) = 0. \] If $n > 2,$ then we have $\R(x_i , x_j)=0,$ and
further $\R(g^ux_i , g^vx_j)=0$ for all $i$ and $j.$ Therefore
$E(n,m)$ has only one coquasitriangular structure which is the
trivial extension of the bicharacter. If $n=2,$ then the Hopf
algebra $E(2,m)$ is known as Nichols' Hopf algebra \cite{n} and its
set of quasitriangular structures were classified by Panaite and Van
Oystaeyen in \cite{pvo}. By the previous equations we can show that
its set of coquasitriangular structures is in one-to-one
correspondence with the set of the matrices $(\R(x_i , x_j))_{m
\times m}.$ By the fact that $E(2,m)$ is self-dual, this coincides
with the result of Panaite and Van Oystaeyen.

\subsection{}
Finally we consider the Hopf quiver $Q=Q(G,g+h)$ where $G$ is the
abelian group $<g> \times <h>.$ Assume that the order of $g$ and $h$
are $m$ and $n$ respectively. A skew-symmetric bicharacter $\R$ of
$G$ is determined by three values $\R(g , g),\R(h , h),\R(g , h).$
As before, we have the equation \[ \R(g , g)^2=1=\R(h , h)^2. \] Let
$(m,n)$ denote the greatest common divisor of $m$ and $n.$ Then the
order of $\R(g,h)$ should satisfy
\[ \ord \R(g,h) | (m,n). \]

In order to have more interesting examples, we assume further that
$m$ and $n$ are even. Take the bicharacter $\R$ such that \[ \R(g ,
g)=-1=\R(h , h).\] Let $q$ denote $\R(g , h).$ Consider the sub
coquasitriangular Hopf algebra, denoted by $H(m,n,q),$ of $(kQ,\R)$
generated by the set of vertices and arrows. Let $x: 1 \To g$ and
$y: 1 \To h$ be the arrows with source $1.$ Then $H(m,n,q)$ is
generated by $g,h,x,y$ satisfying the following relations
\begin{gather*}
g^m=1=h^n, \quad x^2=0=y^2, \quad gx=-xg, \quad hy=-yh, \\
\quad hx=qxh, \quad yg=qgy, \quad yx=qxy.
\end{gather*}
As before, it is not hard to determine the complete list of
coquasitriangular structures on $H(m,n,q).$ we do not repeat the
detail.

\section{classification of finite-dimensional coquasitriangular pointed hopf algebras over an algebraically closed
field of characteristic 0}

From now on our ground field $k$ is assumed to be algebraically
close of characteristic 0. The aim of this section is to give a
complete classification of finite-dimensional coquasitriangular
pointed Hopf algebras over $k.$

\subsection{}%A-S conjecture
Firstly we consider the generation problem of finite-dimensional
coquasitriangular pointed Hopf algebras. In \cite{aeg}
Andruskiewitsch, Etingof and Gelaki proved that finite-dimensional
cotriangular pointed Hopf algebras are generated by their group-like
and skew-primitive elements. It turns out that, with our observation
Corollary 4.3, one can extend their result to more general case in a
fairly straightforward manner.

\begin{proposition}
Suppose that $H$ is a finite-dimensional coquasitriangular pointed
Hopf algebra over $k.$ Then as an algebra $H$ is generated by its
group-like and skew-primitive elements.
\end{proposition}

\begin{proof}
To prove that $H$ is generated by its group-like and skew-primitive
elements, it suffices to prove this is the case for its graded
version $\gr H.$ By Corollary 4.3, the Hopf algebra $\gr H$ is
coquasitriangular and is a crossed product of a cotriangular one by
a group algebra. Note that the cotriangular one is the principal
block of $\gr H,$ while other blocks are obtained by multiplying
group-like elements. Now apply the theorem of Andruskiewitsch,
Etingof and Gelaki \cite{aeg} (Theorem 6.1), we can say that $\gr H$
is generated by its group-likes and the skew-primitives of its
principal block. The proposition follows immediately from this. The
proof is completeed.
\end{proof}

\subsection{}%condition of finite-dimensionality
Next we study a general finite-dimensional coquasitriangular pointed
Hopf algebras $(H,\R)$ via its quiver setting. We may assume without
loss of generality that $H$ is connected. Let $G$ denote its group
of group-like elements. The quiver $Q$ of $H$ is assumed to be
$Q(G,R)$ with ramification datum $R=\sum_{i=1}^t R_ig_i.$ Here the
$R_i$'s are assumed to be positive integers if $t \ge 1.$

\begin{proposition}
Keep the above assumptions and notations. If $G$ is the unit group
$\{1\},$ then $H \cong k.$ When $G \ne \{1\},$ we have
\begin{enumerate}
  \item $t \ge 1$ and the order of $g_i$ is even for all $i.$
  In particular the order of $G$ is even.
  \item $\R$ is a skew-symmetric bicharacter of the group $G$ (by restriction) and $\R(g_i , g_i)=-1$ for all $i.$
  \item $(\gr(H), \gr(\R))$ is isomorphic to the sub
  coquasitriangular Hopf algebra of $(kQ,\mathfrak{R})$ generated by
  vertices and arrows, where $\mathfrak{R}$ is obtained by trivial
  extension of the bicharacter $\R$ of $G.$
\end{enumerate}
\end{proposition}

\begin{proof}
By Theorem 4.2, we can view $(\gr(H),\gr(\R))$ as a large sub
coquasitriangular Hopf algebra of $(kQ,\mathfrak{R}).$ Therefore
$\gr(H)$ contains the sub Hopf algebra generated by vertices and
arrows of $Q.$

If $G=\{1\},$ and suppose that there was an arrow $x$ in $Q.$ Then
by the multiplication formula (2.10) it is easy to verify that the
sub Hopf algebra generated by $x$ is actually isomorphic to the
polynomial algebra in one variable which is of course
infinite-dimensional. This is absurd since we assume the dimension
of $H,$ hence of $\gr(H),$ is finite. Therefore the quiver $Q$ is a
single vertex and so $H$ must be $k.$

In the following we assume that $G \ne \{1\}.$ Since the quiver $Q$
is assumed to be connected, so there are arrows and therefore the
ramification datum $R$ is not 0. It follows that $t \ge 1.$ By
Corollary 4.3, $(\gr(H),\gr(\R))$ is cotriangular, therefore
$\gr(\R)$ is a skew-symmetric bicharacter of $G$ which is the
restriction of $\R.$ This implies that, by restriction, $\R$ is a
skew-symmetric character of $G.$ By the assumption $R=\sum_{i=1}^t
R_ig_i,$ we know the quiver $Q$ contain the sub Hopf quivers
$Q(<g_i>,g_i)$ which are basic cycles. By Subsection 5.1, the sub
coquasitriangular Hopf algebra of $kQ(<g_i>,g_i)$ generated by
vertices and arrows is finite-dimensional if and only if $\R(g_i ,
g_i)=-1.$ This forces that the order of $g_i$ must be even. It is
immediate by Largrange's theorem that the order of $G$ is even. Now
we have proved (1) and (2). For (3), just note that the space of
$\gr(H)$ spanned by group-likes and skew-primitives is equivalent to
the space spanned by vertices and arrows of $Q.$ Now by Proposition
6.1 $\gr(H)$ is generated by its group-like and skew-primitive
elements, hence it is contained in the sub algebra of $kQ$ generated
by vertices and arrows. Along with the first paragraph, we have (3).
We are done.
\end{proof}

\subsection{}%graded coquasitriangular pointed Hopf algebras
As a direct consequence we give an explicit description for
coradically graded finite-dimensional coquasitriangular pointed Hopf
algebras via generators and defining relations.

\begin{corollary}
Let $H$ be a coradically graded finite-dimensional pointed Hopf
algebra with coquasitriangular structure $\R.$ Assume that it is
connected and its group of group-likes $G$ is not the unit group.
Then there exist a generating set $\{ g_i | i=1, \cdots, t \}$ of
the group $G$ and for each $g_i$ a set of $(1,g_i)$-primitive
elements $\{ x_{i,i'} | i'=1, \cdots, R_i \}$ such that $H$ is
generated by $\{g_i,x_{i,i'} | i=1, \cdots, t; i'=1, \cdots, R_i \}$
with defining relations
\begin{gather}
\text{the defining relations of $G$ by the $g_i's,$} \\
x_{j,j'}g_i=\R(g_i , g_j)g_ix_{j,j'}, \ \forall i,j,j', \\
x_{i,i'}^2=0, \ \forall i, i', \\
x_{j,j'}x_{i,i'}=\R(g_i , g_j)x_{i,i'}x_{j,j'}, \ \forall i,i',j,j'.
\end{gather}
Moreover, $H$ has a PBW type basis
\[\{gx_{1,1}^{\sigma_{1,1}} \cdots x_{1,R_1}^{\sigma_{1,R_1}} \cdots
x_{t,1}^{\sigma_{t,1}} \cdots x_{t,R_t}^{\sigma_{t,R_t}} | g \in G,
\sigma_{i,i'}=0,1 \}.\]
\end{corollary}

\begin{proof}
Note that Proposition 3.4 can be established for our case directly,
and so there is no harm to assume that $\R$ concentrates at degree
0. By Proposition 6.2, there is a Hopf quiver $Q=Q(G,R)$ with
$R=\sum_{i=1}^t R_ig_i$ such that $H$ is isomorphic to the sub Hopf
algebra of the coquasitriangular Hopf algebra $kQ,$ determined by
the bicharacter $\R$ as in Theorem 3.3, generated by vertices and
arrows. Denote by $\{ x_{i,i'} | i'=1, \cdots, R_i \}$ the set of
arrows with source $1$ and target $g_i$ which are
$(1,g_i)$-primitive elements. Since any arrow of $Q$ can be obtained
as the ones with source $1$ multiplied by group-like elements, hence
$H$ is generated by $\{g_i,x_{i,i'} | i=1, \cdots, t; i'=1, \cdots,
R_i \}.$ Just as the arguments in Section 5, it is easy to verify
that they satisfy the relations (6.1)-(6.4). So as a linear space
$H$ is spanned by the set
\[\{gx_{1,1}^{\sigma_{1,1}} \cdots x_{1,R_1}^{\sigma_{1,R_1}} \cdots
x_{t,1}^{\sigma_{t,1}} \cdots x_{t,R_t}^{\sigma_{t,R_t}} | g \in G,
\sigma_{i,i'}=0,1 \}.\] Note that this set is linearly independent,
as is clear by straightforward computation using the multiplication
formula (2.10). On the other hand, by a standard application of
Bergman's diamond lemma \cite{b}, the algebra generated by
$\{g_i,x_{i,i'} | i=1, \cdots, t; i'=1, \cdots, R_i \}$ with
defining relations (6.1)-(6.4) has the aforementioned set as a
basis. It follows that (6.1)-(6.4) are sufficient defining relations
for $H.$ Now the proof is finished.
\end{proof}

\subsection{}%lifting and classification
With a help of the lifting theorem of Andruskiewitsch and Schneider
\cite{as1} (Theorem 5.5, see also \cite{as3}), now we are ready to
classify finite-dimensional coquasitriangular pointed Hopf algebras.

\begin{theorem}
Let $(H,\R)$ be a finite-dimensional coquasitriangular pointed Hopf
algebra. Assume that it is connected and its group of group-likes
$G$ is not the unit group. Then there exist a generating set $\{ g_i
| i=1, \cdots, t \}$ of the group $G$ and for each $g_i$ a set of
$(1,g_i)$-primitive elements $\{ x_{i,i'} | i'=1, \cdots, R_i \}$
such that $H$ is generated by $\{g_i,x_{i,i'} | i=1, \cdots, t;
i'=1, \cdots, R_i \}$ with defining relations
\begin{gather}
\text{the defining relations of $G$ by the $g_i's,$} \\
x_{j,j'}g_i=\R(g_i , g_j)g_ix_{j,j'}, \ \forall i,j,j', \\
x_{i,i'}^2=\mu_{i,i'}(1-g_i^2), \ \forall i, i', \\
x_{j,j'}x_{i,i'}-\R(g_i ,
g_j)x_{i,i'}x_{j,j'}=\lambda_{i,i',j,j'}(1-g_ig_j), \ \forall (i,i')
\ne (j,j').
\end{gather}
Here $\mu_{i,i'}$ and $\lambda_{i,i',j,j'}$ are some appropriate
constants in the field $k.$ Moreover, $H$ has a PBW type basis
\[\{gx_{1,1}^{\sigma_{1,1}} \cdots x_{1,R_1}^{\sigma_{1,R_1}} \cdots
x_{t,1}^{\sigma_{t,1}} \cdots x_{t,R_t}^{\sigma_{t,R_t}} | g \in G,
\sigma_{i,i'}=0,1 \}.\]
\end{theorem}

\begin{proof}
By assumption $(\gr(H),\gr(\R))$ satisfies the condition of
Corollary 6.3, hence it can be presented by generators with
relations as given there. Now the theorem follows from the lifting
theorem of Andruskiewitsch and Schneider. For the convenience of the
reader, we include a detailed proof.

Clearly, the coradical $H_{0}$ is the sub Hopf algebra generated by
$\{g_{i}|i=1,\cdots, t\}$.
 By Lemma 6.1 in \cite{as3}, there are elements $\{ x_{i,i'} | i'=1, \cdots, R_i
\}\in H$  corresponding to $\{ x_{i,i'} | i'=1, \cdots, R_i \} \in
\gr(H)$ and they satisfy
$$\Delta(x_{i,i'})=g_{i}\otimes x_{i,i'}+x_{i,i'}\otimes 1,\;\;x_{j,j'}g_i=\R(g_i , g_j)g_ix_{j,j'}$$
for any $i,i',j,j'$. Since $\gr(H)$ is generated by group-like
elements and skew primitive elements, we know $H$ is indeed
generated by $\{g_i,x_{i,i'} | i=1, \cdots, t; i'=1, \cdots, R_i
\}$. For any $i,i',j,j'$, direct computations show that the element
$x_{i,i'}^{2}$ is a $(1, g_{i}^{2})$-primitive element  and
 $x_{j,j'}x_{i,i'}-\R(g_i ,
g_j)x_{i,i'}x_{j,j'}$ is a  $(1,g_{i}g_{j})$-primitive element.

\textbf{Claim.} \emph{For any $i,j\in \{1,\cdots,t\}$, there are no
non-trivial $(1,g_{i}^{2})$-primitive elements and no non-trivial
$(1,g_{i}g_{j})$-primitive elements.}

Here the trivial $(1,g)$-primitive elements is define to be the
elements belonging to the space spanned by $1-g$. Otherwise, assume
$y_{i}$ is a non-trivial $(1,g_{i}^{2})$-primitive element. By
$g_{i}x_{i,i'}g_{i}^{-1}=\R(g_{i},g_{i})x_{i,i'}=-x_{i,i'}$,
$g_{i}^{2}x_{i,i'}g_{i}^{-2}=\R(g_{i},g_{i}^{2})x_{i,i'}=x_{i,i'}$.
Thus $\R(g_{i},g_{i}^{2})=1$ and so
\[\R(g^{2}_{i},g_{i})\R(g^{2}_{i},g_{i})=\R(g^{2}_{i},g^{2}_{i})=1.\]
Note that we always have
$g_{i}^{2}y_{i}g_{i}^{-2}=\R(g_{i}^{2},g_{i}^{2})y_{i}=-y_{i}$ which
implies $\R(g_{i}^{2},g_{i}^{2})=1.$ It's a contradiction.
Similarly, assume that $z$ is a non-trivial
 $(1,g_{i}g_{j})$-primitive element, then
 \begin{eqnarray*}
g_{i}g_{j}z(g_{i}g_{j})^{-1}&=& \R(g_{i}g_{j},g_{j})
\R(g_{i}g_{j},g_{i})z\\
&=& \R(g_{j},g_{j}) \R(g_{i},g_{i}) \R(g_{i},g_{j})
\R(g_{j},g_{i})z\\
&=&(-1)^{2}z=z.
 \end{eqnarray*}
Here we have used the fact that $\R$ is skew-symmetric. Thus this
also implies that $ \R(g_{i}g_{j},g_{i}g_{j})=1.$ This contradicts
to $\R(g_{i}g_{j},g_{i}g_{j})=-1$  by Proposition 6.2 (2).

Therefore, by the claim above, there are
$\mu_{i,i'},\lambda_{i,i',j,j'}\in k$ such that
$$x_{i,i'}^2=\mu_{i,i'}(1-g_i^2),\;\;x_{j,j'}x_{i,i'}-\R(g_i ,
g_j)x_{i,i'}x_{j,j'}=\lambda_{i,i',j,j'}(1-g_ig_j).$$ Consider a
Hopf algebra $A$ which is generated by group-like elements $\{ g_i |
i=1, \cdots, t \}$ and for each $g_i$ a set of $(1,g_i)$-primitive
elements $\{ x_{i,i'} | i'=1, \cdots, R_i \}$ and assume it
satisfies the relations (6.5)-(6.8). The preceding discussions imply
that there is a surjective Hopf algebra map from $A$ to $H$. By
comparing the dimensions, this surjective map is indeed an
isomorphism.
\end{proof}

We remark that, by carrying out the the same computational process
as the examples in Subsections 5.3 and 5.4, it is not difficult to
determine all the possible coquasitriangular structures of the
coquasitriangular pointed Hopf algebras we have just classified.
This is left for the interested reader.

\vskip 0.5cm

\noindent{\bf Acknowledgements:} The research was supported by the
NSF of China (10601052, 10801069). Part of the work was done while
the first author was visiting the Abdus Salam International Centre
for Theoretical Physics (ICTP). He expresses his sincere gratitude
to the ICTP for its support. Both authors would like to thank the
DAAD for financial support which enables them to visit the
University of Cologne.

\end{document}